%% file: main.tex
\documentclass[pamm,a4paper,fleqn]{w-art}
\input{header}

\usepackage{times,cite,w-thm}
\usepackage[T1]{fontenc}
\usepackage[utf8]{inputenc}
\usepackage{graphicx}

\begin{document}

\TitleLanguage[EN]
\title{Graph-to-local limit for a  multi-species nonlocal cross-interaction system}

\author{\firstname{Antonio} \lastname{Esposito}\inst{1,}%
\footnote{e-mail \ElectronicMail{antonio.esposito@maths.ox.ac.uk}.}} 
\address[\inst{1}]{\CountryCode[EN] Mathematical Institute, University of Oxford, Woodstock Road, Oxford, OX2 6GG, United Kingdom}

\author{\firstname{Georg} \lastname{Heinze}\inst{2,}%
     \footnote{Corresponding author. E-mail \ElectronicMail{ georg.heinze@uni-a.de}.}}
\address[\inst{2}]{\CountryCode[DE] Faculty of Mathematics, University of Augsburg, Universitätsstraße 12a, 86159 Augsburg, Germany}

\author{\firstname{Jan-Frederik} \lastname{Pietschmann}\inst{2,}%
\footnote{e-mail \ElectronicMail{jan-f.pietschmann@uni-a.de}.}
     }
\author{\firstname{André} \lastname{Schlichting}\inst{3,}%
     \footnote{e-mail \ElectronicMail{a.schlichting@uni-muenster.de}.}
     }
\address[\inst{3}]{\CountryCode[DE] Institute for Analysis and Numerics, University of Münster, Orléans-Ring 10, 48149 Münster, Germany}

\AbstractLanguage[EN]
\begin{abstract}
    In this note we continue the study of nonlocal interaction dynamics on a sequence of infinite graphs, extending the results of ~\cite{EHS23} to an arbitrary number of species. 
    Our analysis relies on the observation that the graph dynamics form a gradient flow with respect to a non-symmetric Finslerian gradient structure. Keeping the nonlocal interaction energy fixed, while localising the graph structure, we are able to prove evolutionary $\Gamma$-convergence to an Otto-Wassertein-type gradient flow with a tensor-weighted, yet symmetric, inner product. 
    As a byproduct this implies the existence of solutions to the multi-species non-local (cross-)interacation system on the tensor-weighted Euclidean space.
\end{abstract}
\maketitle

\section{Introduction}

In this work we consider the local limit of a system of nonlocal (cross-)interaction equations on graphs, introduced in \cite{HPS21}. We mildly extend the recent work \cite{EHS23}, which established the local limit for a single species by introducing $N-1$ additional species and a coupling between different species via the nonlocal interaction energy functional. 

Following~\cite{EPSS2021,HPS21, EspositoPatacchiniSchlichting2023}, our model is set on an abstract graph, where vertices are represented by a Radon measure $\mu\in \calM^+(\Rd)$, which we call the \emph{base measure}. Intuitively, the support of $\mu$ defines the underlying set of vertices, i.e. $V = \supp\mu$. In particular, any finite graph can be represented by choosing $\mu=\mu^n=\sum_{i=1}^n\delta_{x_i}/n$, for $x_1, x_2,\dots,x_n\in\Rd$. Vertices are connected according to a given \emph{edge weight map} $\eta: \Rddiag \to [0,\infty)$, with $\Rddiag\coloneqq\set{(x,y)\in\Rd\times\Rd:x\ne y}$. Consequently, the set of edges is given as $G=\set*{ (x,y)\in \Rddiag: \eta(x,y)>0}$ and mass from a vertex $x\in \Rd$ can be nonlocally transported to~$y\in \Rd$ along a channel with capacity/weight given by $\eta(x,y)$. We remark that the model also extends to multiple graphs $(\mu^{(i)},\eta^{(i)})$, $i=1,\ldots,N$, one for each species, but we consider for simplicity in the sequel only a single graph $(\mu,\eta)$.
As our notion of graph allows for isolated points,  the local gradient and divergence operators are replaced by nonlocal analogues. Precisely, we define the nonlocal gradient of a function $f:\Rd\to\R$ and the nonlocal divergence is then given for a signed Radon measure $j\in\calM(G)$ by
\begin{align}\label{eq:NL-grad_NL-div_intro}
    \babla f(x,y) \coloneqq f(y)-f(x)\quad\text{and}\quad
    \babla\cdot j(\dx x) \coloneqq \frac12\int_{\Rd\setminus\!\set*{x}} \eta(x,y) \dx( j(x,y)-j(y,x)),
\end{align}

Another structural property of a graph is that, in contrast to a local setting, probability densities are defined on vertices, while fluxes and velocities are defined on edges. In order to cope with this fact, one needs to introduce a suitable interpolating function into the kinetic relation and hence the dynamics. In the present work we employ an upwind interpolation. We refer to \cite{EPSS2021,EspositoPatacchiniSchlichting2023,HPS21} for more details. Another important aspect is to deal with a large number of entities, for instance individuals or data; hence it is crucial to consider discrete and continuum models. The setup introduced for a single species in~\cite{EPSS2021} and extended to two species in \cite{HPS21} allows to consider both descriptions in a unified framework, as follows.

The nonlocal (cross-)interaction system of $N\in\bbN$ equations we consider can be specified through three elements: a \emph{nonlocal continuity equation}, an \emph{upwind flux interpolation}, and a \emph{constitutive relation} for a nonlocal velocity. 
The nonlocal continuity equation describes the time-evolution of an $N$-tuple of probability measures $\rhoup_t =(\rho_t^{(1)},\ldots,\rho_t^{(2)}) \in (\calP(\Rd))^N$, for $t$ in a time interval $[0,T]$ via the equations
\begin{subequations}\label{eq:intro:NLNL}
\begin{equation}\label{eq:intro:NL-CE}
  \partial_t \rho^{(i)}_t + \babla\cdot j^{(i)}_t = 0, \qquad i=1,\ldots,N.
\end{equation}
Here, the flux is a time-dependent pair of antisymmetric measures, $\jup_t = (j^{(1)}_t,\ldots,j^{(N)}_t) \in (\calM(G))^N$. We will use the shorthand notation $(\bs \rhoup,\jupbold) =((\rhoup_t)_t,(\jup_t)_t)\in \NCE_T$ for any solution of~\eqref{eq:intro:NL-CE}, cf.~Definition~\ref{def:NL-metric}. Given a two-species nonlocal time-dependent velocity field $\vupbold = (\bs v^{(1)},\ldots,\bs v^{(N)}) =((v_t^{(1)})_t,\ldots,(v_t^{(N)})_t): [0,T]\times G\to \R^N$ the kinetic relation constituting the associated flux is given by the upwind interpolation
\begin{equation}\label{eq:intro:NL-flux}
	\dx j^{(i)}_t(x,y) = v^{(i)}_t(x,y)_+\dx(\rho^{(i)}\otimes\mu)(x,y)-v^{(i)}_t(x,y)_-\dx(\mu\otimes\rho^{(i)})(x,y),\qquad  i=1,\ldots,N, 
\end{equation}
where for $a\in \R$ we denote by $a_+ = \max\set*{a,0}$ and $a_- = \max\set*{-a,0}$ the positive and negative part, respectively. The last element of the model is the identification of the velocity fields in terms of symmetric interaction potentials $K^{(ik)}:\Rd\times\Rd \to \R$, $i,k=1\ldots,N$ and potentials $P^{(i)}:\Rd\to \R$, $i=1,\ldots,N$ by
\begin{equation}\label{eq:intro:NL-velocity}
  v^{(i)}_t(x,y) = - \babla P^{(i)}(x,y)- \sum_{k=1}^N\babla K^{(ik)}\ast \rho^{(k)}_t(x,y),\qquad i=1,\ldots,N.
\end{equation}
\end{subequations}
System~\eqref{eq:intro:NLNL} was introduced in~\cite{HPS21} as a \textit{Finslerian} gradient flow of the nonlocal cross-interaction energy 
\begin{equation}\label{eq:interaction_energy}
	\calE(\rho)=\sum_{i=1}^N\int_\Rd P^{(i)}(x)\dx\rho^{(i)}(x)+\frac{1}{2}\sum_{i,k=1}^N\iint_{\R^{2d}} K^{(ik)}(x,y)\dx\rho^{(k)}(y)\dx\rho^{(i)}(x). 
\end{equation}
Note that the velocity field is given as the nonlocal gradient of the first variation of the energy, that is $v^{(i)}_t = -\babla \frac{\delta\calE(\rho_t)}{\delta\rho^{(i)}}$, where $\frac{\delta\calE(\rho_t)}{\delta\rho^{(i)}} = P^{(i)} +\sum_{k=1}^2 K^{(ik)} \ast \rho^{(k)}$ denotes the variational derivative of $\calE$ with respect to the species $i$ and $(K^{(ik)}\ast \rho^{(k)})(x) = \int_\Rd K^{(ik)}(x,y)\dx \rho^{(k)}(y)$, for any $x\in\Rd$, $i,k=1,\ldots,N$.
In the case $\rho^{(1)},\ldots,\rho^{(N)}\ll\mu$, denoting by~$\rho^{(i)}$  both the measure~$\rho^{(i)}$ and its density with respect to $\mu$, system \eqref{eq:intro:NLNL} reads for $\mu-$a.e. $x$ as
\begin{equation}\label{eq:nlnl-interaction-eq}\tag{$\mathsf{NL^2IE}$}
    \begin{split}
        \partial_t\rho_t^{(i)}(x)&+\int_\Rd \bigg(\babla P^{(i)}(x,y)+\sum_{k=1}^N\babla K^{(ik)}*\rho^{(k)}_t(x,y)\bigg)_- \eta(x,y) \rho^{(i)}_t(x) \dx\mu(y)\\
        &- \int_\Rd \bigg(\babla P^{(i)}(x,y)+\sum_{k=1}^N\babla K^{(ik)}*\rho^{(k)}_t(x,y)\bigg)_+ \eta(x,y) \rho_t^{(i)}(y)\dx\mu(y)=0,
    \end{split} \qquad i=1,\ldots,N.
\end{equation}
We will refer to it as nonlocal-nonlocal in view of the nonlocal nature of the graph.
An intriguing problem is to understand the limiting behaviour of weak solutions to~\eqref{eq:nlnl-interaction-eq} as the graph structure localises, i.e. the range of connection between vertices decreases, while the weight of each connecting edge increases. 
Following a formal argument presented in~\cite[Section 3.5]{EPSS2021}, one expects that these weak solutions approximate, under suitable conditions on the interaction kernels, weak solutions of the nonlocal cross-interaction equation on~$\Rd$. However, as we shall see, the intrinsic geometry of the graph impacts the limiting gradient structure of the equation.
Accordingly, the main goal of this work is to provide a rigorous proof of the local limit of the \eqref{eq:nlnl-interaction-eq} along a sequence of edge weight maps $\eta^\eps:\Rddiag\to[0,\infty)$ defined by
\begin{align}\label{eq:intro:etaeps}
	\eta^\eps(x,y) \coloneqq \frac{1}{\eps^{d+2}}\vartheta\bra*{\frac{x+y}{2},\frac{x-y}{\eps}},
\end{align} 
in terms of a reference connectivity $\vartheta:\Rd\!\setminus\!\set{0}\to [0,\infty)$ satisfying the Assumptions~\eqref{theta1}~--~\eqref{theta4} below. The corresponding set of edges is then denoted $G^{\!\!\:\eps}\coloneqq\set{(x,y)\in\Rddiag:\eta^\eps(x,y)>0}$. The scaling in~\eqref{eq:intro:etaeps} leads to the local evolution
\begin{equation}\label{eq:intro:NLIE:one}\tag{$\mathsf{NLIE}_\bbT$} 
    \partial_t\rho^{(i)}_t=\nabla\cdot\bigg(\rho^{(i)}_t  \bbT\bigg( \sum_{k=1}^N \nabla P^{(i)}+\nabla K^{(ik)}*\rho^{(k)}_t\bigg)\bigg), \qquad i=1,\ldots,N, 
\end{equation}
where the \emph{tensor} $\bbT:\Rd \to \R^{d\times d}$ depends on the nonlocal structure encoded by $\mu$ and $\vartheta$. The limiting equation \eqref{eq:intro:NLIE:one} can be similarly decomposed into three components. First, the  \emph{local continuity equation} on $\Rd$ given as
\begin{subequations}\label{eq:intro:NLIE:sub}
\begin{equation}\label{eq:intro:CE}
	\partial_t \rho^{(i)}_t + \nabla\cdot \hat\jmath^{(i)}_t = 0, \qquad i=1,\ldots,N, 
\end{equation}
where now $\jhup_t = (\jh^{(1)}_t,\ldots,\jh^{(N)}_t)\in (\calM(\Rd;\Rd))^N$ is a time-dependent vector-valued flux.
Second, a \emph{kinetic relation}, for the flux $\jhup_t$ in terms of a vector field $\hat \vup_t = (\hat v^{(1)}_t,\ldots,\hat v^{(N)}_t): \Rd \to \R^{Nd}$ encoding the tensor structure of~\eqref{eq:intro:NLIE:one} as
\begin{equation}\label{eq:intro:fluxtensor}
	\hat\jmath^{(i)}_t(\dx x) = \rho^{(i)}_t(\dx x) \bbT(x) \hat v^{(i)}_t(x) = \rho_t(\dx x) \sum_{l,m=1}^d \bbT_{lm}(x) \hat v^{(i)}_{t,m}(x)e_l, \qquad i=1,\ldots,N . 
\end{equation}
Third, a  \emph{constitutive relation} for the velocity linking to the interaction energy~\eqref{eq:interaction_energy} given by
\begin{equation}\label{eq:intro:velocity}
	\hat v^{(i)}_t = - \nabla P^{(i)}- \sum_{k=1}^N \nabla K^{(ik)} \ast \rho^{(k)}_t  = -\nabla \frac{\delta \calE(\rho_t)}{\delta \rho^{(i)}}, \qquad i=1,\ldots,N .
\end{equation}
\end{subequations}
We provide a rigorous proof for the convergence of \eqref{eq:nlnl-interaction-eq} with $\eta^\eps$ given through~\eqref{eq:intro:etaeps} and $\mu\ll\scrL^d$ satisfying \eqref{mu1}, \eqref{mu2} below to \eqref{eq:intro:NLIE:one}, in case of $C^1$ interaction kernels $K^{(ik)}$, $i,k=1,\ldots,N$ satisfying assumptions \eqref{K1}~--~\eqref{K4} below. Note that the regularity assumptions on $\mu$ are far less restrictive than it seems at first glance, since such $\mu$ are known to be approximated by finite graphs, cf. \cite{EPSS2021} in conjunction with \cite[Appendix B]{EHS23}.
The result is also somewhat sharp regarding the $C^1$-regularity of the kernels $K^{(ik)}$, $i,k=1,\ldots,N$, since for attractive pointy potentials one cannot expect convergence of weak solutions, as pointed out in~\cite[Remark 3.18]{EPSS2021}.

\section{Assumptions and definitions}\label{sec:preliminaries}

\paragraph{Graph}

We state the assumptions on the base measure $\mu\in\Mplus(\Rd)$ and the edge connectivity $\vartheta:\Rd\times(\Rd\!\setminus\!\set*{0})\to[0,\infty)$. We assume $\dx\mu = \widetilde\mu\dx\scrL^d$ such that there exists $\omega_\mu\in C([0,\infty);[0,\infty))$ with $\omega_\mu(\delta)\to 0$ as $\delta\to 0$ and
\begin{align}
	\label{mu1}\tag{$\muup_1$} 
    &\forall\, x,y\in\Rd \text{ it holds } \abs{\widetilde\mu(x)-\widetilde\mu(y)}\le\omega_\mu(\abs{x-y}),\\
	\label{mu2}\tag{$\muup_2$} 
    &\exists\, c_\mu,C_\mu>0 \text{ such that } \forall x\in\Rd \text{ it holds } c_\mu \le\widetilde\mu(x) \le C_\mu.
\end{align}
For the edge connectivity map $\vartheta:\Rd\times(\Rd\!\setminus\!\set*{0})\to[0,\infty)$ we assume there exists a modulus of continuity $\omega_\vartheta\in C([0,\infty);[0,\infty))$ satisfying $\omega_\vartheta(0)=0$, such that
\begin{align}
	&\label{theta1}\tag{$\varthetaup_1$} \forall z\in\Rd \text{ the map } w\mapsto\vartheta(z,w) \text{ is symmetric and continuous on } \set{\vartheta>0};\\
    &\label{theta_new}\tag{$\varthetaup_2$} 
    \forall z,\bar z \in\Rd, w\in\Rd\!\setminus\!\set*{0} \text{ it holds } \abs{\vartheta(z,w)-\vartheta(\bar z,w)}\le \omega_\vartheta(\abs{z-\bar z}); \\
	&\label{theta2}\tag{$\varthetaup_3$}
	\exists\!\: \Csupp>0 \text{ such that }\forall z\in\Rd \text{ it holds } \supp\vartheta(z,\cdot) \subset B_{\Csupp};\\
	&\label{theta3}\tag{$\varthetaup_4$}
	\exists\!\: \Ctheta>0 \text{ such that }\sup_{(z,w)\in\Rd\times(\Rdzero)} \abs{w}^2\vartheta(z,w) \le \Ctheta;\\
	&\label{theta4}\tag{$\varthetaup_5$}
	\exists\!\: \ctheta>0 \text{ such that } \forall z,\xi\in\Rd \text{ it holds} \int_{\Rdzero}\! \abs{w\cdot\xi}^2\vartheta(z,w)\dx w \ge \ctheta \abs{\xi}^2.
\end{align}

\begin{example}
     Consider a tensor $\mathbb{D}\in C(\Rd ; \R^{d \times d})$ uniformly elliptic and bounded, i.e. there exist $0<D_*\le D^* < +\infty$ such that $D_* \Id\leq \mathbb{D} \leq D^* \Id$ in the sense of quadratic forms. We define the edge connectivity
    \begin{align*}
        \vartheta(z,w)=\begin{cases} 
    d(z), & \skp{w,\mathbb{D}(z)^{-1} w}\leq 1 ,\\ 
    0, & \skp{w,\mathbb{D}(z) w}> 1,
    \end{cases} 
    \end{align*}
    where $d(z) = 2/\pra{C_d \bra*{\det \mathbb{D}(z)}^{\frac{1}{2}}}$ and $C_d = \pi^{\frac{d}{2}} /(2\Gamma(\frac{d}{2}+2))$. Then, with $\mu =\scrL^d$ we obtain $\bbT = \bbD$, cf. \cite[Section~2.2]{EHS23}. In particular, for $\mu=\scrL^d$ and $\vartheta(z,w) = \vartheta(w) = \tilde C_d\bbone_{B_1}(w)$, with a suitable dimensional constant $\tilde C_d>0$, we have $\bbT = \Id$.
\end{example}

\paragraph{Nonlocal interaction energy}
	To simplify notation, we drop in what follows the potential energy terms $P^{(i)}$, $i=1,\ldots,N$, since they can be treazed similarly. 
 The remaining two-species nonlocal interaction energy considered in what follows is defined by
	\begin{align*}
		\calE(\rhoup) \coloneqq \frac12\sum_{i,k=1}^N\iint_{\R^{2d}} K^{(ik)}(x,y)\dx\rho^{(i)}(x)\dx\rho^{(k)}(y).
	\end{align*}
	We assume the interaction kernels $K^{(ik)}\colon\Rd\times\Rd\to\R$, $i,k=1,\ldots,N$ to satisfy the following assumptions:
	\begin{align}
		\label{K1}\tag{\textsf{K1}}
		&K^{(ik)}\in C^1(\Rd\times\Rd);\\
		\label{K2}\tag{\textsf{K2}}
		&K^{(ik)}(x,y)=K^{(ik)}(y,x)\quad \text{ for  } (x,y)\in\Rd\times\Rd;\\
		\label{K3}\tag{\textsf{K3}}
		\begin{split}
			&\exists\, L_K>0 \text{ such that for all }(x,y),(x',y')\in\Rd\times\Rd, i,k=1,2 \text{ it holds}\\
			&|K^{(ik)}(x,y)-K^{(ik)}(x',y')|\le L_K\left(|(x,y)-(x',y')|\vee|(x,y)-(x',y')|^2\right);
		\end{split}\\
		\label{K4}\tag{\textsf{K4}}
		&\exists\, C_K>0 \text{ such that for all }(x,y)\in\Rd\times\Rd \text{ it holds }|\nabla K^{(ik)}(x,y)|\le C_K(1+|x|+|y|).
	\end{align}
    Furthermore, we require symmetry of the cross-interactions, $K^{(ik)}=K^{(ki)}$ for $i,k =1,\ldots,N$ with $i\ne k$, in order to ensure that \eqref{eq:nlnl-interaction-eq} is a gradient flow of~$\calE$, see~\cite{DiFrancesco2013}. This assumption can be weakened to $\beta^{(ik)}K^{(ik)}=\beta^{(ki)} K^{(ki)}$ for numbers $\beta^{(ik)}>0$ for $i,k =1,\ldots,N$ with $i\ne k$, as is done in \cite{HPS21} for two species.
\paragraph{Graph quasi-metric structure}

In order to have a graph-analogue of Otto-Wasserstein gradient flows, cf. \cite{Otto2001GeometryPME, AmbrosioGigliSavare2008}, for interaction energies we defined a suitable quasi-metric space, where the quasi-distance is obtained in a dynamical formulation à la Benamou--Brenier, \cite{BenamouBrenier2000}. For this reason, it is crucial to identify paths connecting probability measures, a \emph{nonlocal} continuity equation, and an \emph{upwind-induced} action functional to be minimized, resembling the total kinetic energy.
\begin{definition}[Action]\label{def:NL-action}
	Let $\mu\in\calM^+(\Rd)$ and $\eta:\Rddiag\to[0,\infty)$ as before. For $\rhoup\in(\calP(\Rd))^N$ and $\jup\in(\calM(G))^N$, we define the \emph{nonlocal action density} funcional
	\begin{equation*}
        \A(\mu,\eta;\rhoup,\jup)\coloneqq \begin{cases}
            \frac12\sum_{i=1}^N\iint_G \bra[\big]{v^{(i)}_+(x,y)-v^{(i)}_+(y,x)}\eta(x,y) \dx\rho^{(i)}(x)\dx\mu(y), &\text{if \eqref{eq:intro:NL-flux} holds for $(v^{(1)},\ldots,v^{(N)})$},\\
            \infty, &\text{if there is no such $(v^{(1)},\ldots,v^{(N)})$}.
        \end{cases}
	\end{equation*}
	As proved in \cite[Lemma~2.13]{HPS21} for fixed $\eta$ the functional $\calA(\mu,\eta;\rhoup,\jup)$ is weakly-$^\ast$ lower semicontinuous in $\mu$, $\rhoup$ and $\jup$. If $\A(\mu,\eta;\rhoup,\jup)<\infty$, we denote $\tA(\mu,\eta;\rhoup,\vup)\coloneqq \A(\mu,\eta;\rhoup,\jup)$ the action density with the corresponding velocity $\vup = (v^{(1)},\ldots,v^{(N)})$. Given a pair of curves $(\bs\rhoup,\jupbold)\coloneqq ((\rhoup_t)_{t\in[0,T]},(\jup_t)_{t\in[0,T]})$ with $\rhoup_t\in(\calP(\Rd))^N$ and $\jup_t\in(\calM(G))^N$, we define the \emph{nonlocal action} of $(\bs\rhoup,\jupbold)$ as $\bs\A(\mu,\eta;\bs\rhoup,\jupbold)\coloneqq \int_0^T\A(\mu,\eta;\rhoup_t,\jup_t)\dx t$.
\end{definition} 
\begin{definition}[Nonlocal continuity equaton]\label{def:NL-CE}
We denote by $\NCE_T$ the set of all weak solutions $(\bs\rhoup,\jupbold)$ to the \emph{nonlocal continuity equation} \eqref{eq:intro:NL-CE}. More precisely, these are pairs such that $\bs\rhoup:[0,T]\to(\calP(\Rd))^N$ is a narrowly continuous curve, $\jupbold:[0,T]\to(\calM(G))^N$ is a measurable curve satisfying $\sum_{i=1}^N\int_0^T\iint_G(2\wedge|x-y|)\eta(x,y)\dx \abs{j^{(i)}_t}(x,y)\dx t<\infty$, and for all $\bs\varphi\in C_c^\infty((0,T)\times\Rd)$ it holds 
\begin{equation*}
	\int_0^T\int_\Rd\partial_t\varphi_t(x)\dx\rho^{(i)}_t(x)\dx t +\frac12\int_0^T\iint_G\babla\varphi_t(x,y)\eta(x,y)\dx j^{(i)}_t(x,y)\dx t=0, \qquad i=1,\ldots,N.
\end{equation*}
\end{definition} 
\begin{definition}[Nonlocal extended quasi-metric]\label{def:NL-metric}
We define the \emph{nonlocal extended quasi-metric} at $\varrhoup_0, \varrhoup_1 \in (\calP_2(\Rd))^N$ by
	\begin{equation*}
		\calT^{\otimes N}_{\mu,\eta}(\varrhoup_0,\varrhoup_1) \coloneqq\inf_{(\bs\rhoup,\jupbold)\in\NCE(\varrhoup_0, \varrhoup_1)}\bra[\big]{\bs\A(\mu,\eta;\bs\rhoup,\jupbold)}^{1/2},
	\end{equation*}
    where $\NCE(\varrhoup_0, \varrhoup_1) \coloneqq \set{(\bs\rhoup,\jupbold)\in\NCE_1: \rhoup_0=\varrhoup_0, \rhoup_1=\varrhoup_1}$. When $\mu$ and $\eta$ are clear from the context, we will shorten notation by writing $\calT^{\otimes N}$. In the context of pairs $(\mu,\eta^\eps)$, we will often write $\calT^{\otimes N}_{\eps}$.
\end{definition}
Properties of $\calT^{\otimes N}_{\mu,\eta}$ can be found in \cite[Section~2]{HPS21}, including that it is indeed an extended quasi-metric on $(\calP_2(\Rd))^N$. Furthermore, we emphasize that since the nonlocal continuity equation does not couple different species and the two-species action is a sum of single-species actions, the quasi-metric defined above is in fact a tensor product of single-species quasi-metrics, i.e. $\calT^{\otimes N}_{\mu,\eta}(\varrhoup_0,\varrhoup_1) = \prod_{i=1}^N\calT_{\mu,\eta}(\varrho^{(i)}_0,\varrho^{(i)}_1)$, being $\calT_{\mu,\eta}$ the single-species nonlocal extended quasi-metric.

\begin{definition}[Absolutely continuous curve and metric derivative]\label{def:AC-curves_metric_derivative}
	We denote by $\AC^2([0,T];((\calP_2(\Rd))^N,\calT^{\otimes N}_{\mu,\eta}))$ the \emph{set of $2$-absolutely continuous curves} with respect to $\calT^{\otimes N}_{\mu,\eta}$, that is for such a \emph{$2$-absolutely continuous curve} $\bs\rhoup$ there exists $m\in L^2([0,T])$ such that for all $0\leq s \leq t \leq T$ we have $\calT^{\otimes N}_{\mu,\eta}(\rhoup_s,\rhoup_t) \le \int_s^t m(\tau)\dx \tau$.
    The \emph{(forward) metric derivative} of a $2$-absolutely continuous curve $\bs\rhoup\in\AC^2([0,T];((\calP_2(\Rd))^N,\calT^{\otimes N}_{\mu,\eta}))$ is defined for a.e. $t\in[0,T]$ by
	\begin{align*}
		\abs{\rhoup_t'}_{\mu,\eta}\coloneqq \lim_{\tau\searrow 0} \frac{\calT^{\otimes N}_{\mu,\eta}(\rhoup_t,\rhoup_{t+\tau})}{\tau} = \lim_{\tau\searrow 0} \frac{\calT^{\otimes N}_{\mu,\eta}(\rhoup_{t-\tau},\rhoup_t)}{\tau}.
	\end{align*}
	As for the extendet quasi-metric, we will often shorten the notation by writing  $\abs{\rhoup_t'}_\eps$ instead of $\abs{\rhoup_t'}_{\mu,\eta^\eps}$.
\end{definition}

\begin{definition}[Metric slope and De Giorgi functional]\label{def:ls-Giorgi}
For any $\rhoup\in(\calP_2(\Rd))^N$, we recall the definition of the variational derivative  $\frac{\delta\calE}{\delta \rhoup}(\rhoup) = (\sum_{k=1}^N K^{(1k)}\ast\rho^{(k)},\ldots,\sum_{k=1}^N K^{(Nk)}\ast\rho^{(k)})$, and we define the \emph{graph (quasi-)metric slope} at $\rhoup$ by
\begin{equation*}
    \calD_{\mu,\eta}(\rho) \coloneqq  \tA\bra[\Big]{\mu,\eta;\rhoup, -\dgrad \frac{\delta\calE}{\delta \rhoup}(\rhoup)}. 
\end{equation*}

With this, for any $\bs\rhoup \in \AC^2([0,T];((\calP_2(\Rd))^N,\calT^{\otimes N}))$, the \emph{graph De Giorgi functional} at $\bs\rhoup$ is defined as
\begin{equation*}
	\bs\calG_{\mu,\eta}(\bs\rhoup)\coloneqq\calE(\rhoup_T)-\calE(\rhoup_0)+\frac{1}{2}\int_0^T\bra[\big]{\calD_{\mu,\eta}(\rhoup_\tau) + |\rhoup_\tau'|_{\mu,\eta}^2}\dx \tau.
\end{equation*}
As before, we will often denote $\calD_{\!\eps} \coloneqq \calD_{\mu,\eta^\eps}$ and $\bs\calG_{\!\!\eps} \coloneqq \bs\calG_{\mu,\eta^\eps}$.
\end{definition}

\paragraph{Tensorized local metric structure}
The limiting local structure is also defined in the spirit of~\cite{BenamouBrenier2000} as follows.
\begin{definition}[Tensorized local action]\label{def:L-action}
    Given a Borel measurable, continuous, symmetric, and uniformly elliptic tensor $\bbT:\Rd\to\Rd\times\Rd$, we define the \emph{local action density} by
    \begin{align*}
        \calA_\bbT(\rhoup,\jhup) \coloneqq \sum_{i=1}^N\norm*{\frac{\dx \jh^{(i)}}{\dx\rho^{(i)}}}_{L^2(\rho^{(i)};\R_\bbT^d)}^2 \coloneqq \sum_{i=1}^N\int_\Rd\skp*{\bbT^{-1}(x)\frac{\dx \jh^{(i)}}{\dx\rho^{(i)}}(x),\frac{\dx \jh^{(i)}}{\dx\rho^{(i)}}(x)}\dx\rho^{(i)}(x).
    \end{align*}
    Similar to the graph setting, we denote $\tA_\bbT(\rhoup,\hat\vup) \coloneqq \A_\bbT(\rhoup,\jhup)$ if $\hat\vup$ and $\jhup$ are related by \eqref{eq:intro:fluxtensor}, and we denote the \emph{local action} of a pair of curves $(\bs\rhoup,\jhupbold)$ by $\bs\calA_\bbT(\bs\rhoup,\jhupbold) \coloneqq \int_0^1\calA_\bbT(\rhoup_t,\jhup_t) \dx t$.
\end{definition}  
\begin{definition}[Local continuity equation]\label{def:L-CE}
    We denote by $\CE_T$ the set of all weak solutions $(\bs\rhoup,\jhupbold)$ to the \emph{local continuity equation} \eqref{eq:intro:CE}. More precisely, these are pairs such that $\bs\rhoup:[0,T]\to(\calP(\Rd))^2$ is a narrowly continuous curve, $\jhupbold:[0,T]\to(\calM(\Rd;\Rd))^N$ is a measurable curve satisfying $\sum_{i=1}^N\int_0^T\abs{j^{(i)}_t}(\Rd)\dx t<\infty$, and for all $\bs\varphi\in C_c^\infty((0,T)\times\Rd)$ it holds 
    \begin{equation*}
	   \int_0^T\int_\Rd\partial_t\varphi_t(x)\dx\rho^{(i)}_t(x)\dx t +\int_0^T\int_\Rd\nabla\varphi_t(x)\cdot\dx \jh^{(i)}_t(x)\dx t=0, \qquad i=1,\ldots,N.
    \end{equation*}
\end{definition}
\begin{definition}[Tensorized local metric]\label{def:L-metric}
    We define the \emph{tensorized local metric} at $\varrhoup_0, \varrhoup_1 \in (\calP_2(\Rd))^N$ by
    \begin{align*}
        W_\bbT^{\otimes N}(\varrhoup_0,\varrhoup_1) \coloneqq\inf_{(\bs\rhoup,\jhupbold)\in\CE(\varrhoup_0, \varrhoup_1)}\bra[\big]{\bs\calA_\bbT(\bs\rhoup,\jhupbold)}^{1/2},
    \end{align*}
    where $\CE(\varrhoup_0, \varrhoup_1) \coloneqq \set{(\bs\rhoup,\jupbold)\in\CE_1: \rhoup_0=\varrhoup_0, \rhoup_1=\varrhoup_1}$.
\end{definition}

We notice that in contrast to the graph quasi-metric, the local metric is indeed symmetric. Furthermore, similar to the graph, we have the decoupling $W_\bbT^{\otimes N}(\varrhoup_0,\varrhoup_1) = \prod_{i=1}^N W_\bbT(\varrho^{(i)}_0,\varrho^{(i)}_1)$, being $W_\bbT$ the single-species tensorized local metric.

\begin{definition}\label{def:deGiorgi_local}
Let $\rhoup\in(\calP_2(\Rd))^N$. The metric slope of the nonlocal interaction energy is given by 
\begin{align*}
        \calD_{\;\!\bbT}(\rhoup) 
        \coloneqq \tA_\bbT\bra[\bigg]{\rhoup,-\bbT\nabla\frac{\delta\calE(\rhoup)}{\delta\rhoup}}
        =\sum_{i=1}^N\int_\Rd\skp[\bigg]{\bra[\bigg]{\sum_{k=1}^N \nabla K^{(ik)}\ast\rho^{(k)}},\bbT(x)\bra[\bigg]{\sum_{k=1}^N \nabla K^{(ik)}\ast\rho^{(k)}}}\dx\rho^{(i)}(x).
    \end{align*}
For any $\bs\rhoup \in \AC^2([0,T];((\calP_2(\Rd_{\bbT}))^N,W_\bbT^{\otimes N}))$, the \emph{local De Giorgi functional} at $\bs\rhoup$ is defined as
\begin{equation*}
	\bs\calG_\bbT(\bs\rhoup)\coloneqq\calE(\rhoup_T)-\calE(\rhoup_0)+\frac{1}{2}\int_0^T\bra[\big]{\calD_{\;\!\bbT}(\rhoup_\tau) + |\rhoup_\tau'|^2_{\bbT}}\dx \tau.
\end{equation*}
\end{definition}

\section{Results}\label{sec:results}

The main result of the present work is the graph-to-local limit for the nonlocal interaction equation.

\begin{theorem}[Graph-to-local limit]\label{thm:main_result}
    Let $(\mu,\vartheta)$ satisfy \eqref{mu1}, \eqref{mu2} and \eqref{theta1}~--~\eqref{theta4}. Let $\eta^\eps$ be given by \eqref{eq:intro:etaeps} and assume $K^{(ik)}$, $i,k=1,\ldots,N$ satisfy \eqref{K1}~--~\eqref{K4}. For any $\eps>0$ suppose that $\bs\rhoup^\eps\in\AC^2([0,T];((\calP_2(\Rd))^N,\calT^{\otimes N}_\eps))$ with $(\rhoup_0^\eps)_\eps \subset (\calP_2(\Rd))^N$ such that $\sup_{\eps>0}\sum_{i=1}^N M_2(\rho_0^{(i),\eps}) < \infty$ is a gradient flow of $\calE$ in $((\calP_2(\Rd))^N,\calT^{\otimes N}_\eps)$, that is
  \begin{equation*}
      \bs\calG_{\!\!\eps}(\bs\rhoup^\eps) = 0 \quad \text{for any } \eps>0.
  \end{equation*}
  Then, there exists $\bs\rhoup \in \AC^2([0,T];((\calP_2(\Rd))^N,W_\bbT^{\otimes N}))$ such that  $\rho_t^{(i),\eps} \rightharpoonup \rho^{(i)}_t$ as $\eps\to0$ for all $t\in[0,T]$, $i=1,\ldots,N$ and $\bs \rhoup$ is a gradient flow of $\calE$ in $((\calP_2(\Rd))^N,W_\bbT^{\otimes N})$, that is
    \begin{equation*}
      \bs\calG_\bbT(\bs\rhoup) = 0,
  \end{equation*}
where the tensor is defined by $\bbT(x) \coloneqq\frac12\int_{\Rdzero} w\!\otimes\! w \tilde\mu(x) \vartheta(x,w)\dx w$.
\end{theorem}
\begin{proof}
    We observe a coupling of the different species only at the level of the energy and metric slopes, but not on the level of the metrics or metric derivatives. Hence, it suffices to consider a singe species, when studying the metric structures.

    Following the strategy of \cite{EHS23}, we recall that $\bs\rhoup^\eps\in\AC^2([0,T];((\calP_2(\Rd))^N,\calT^{\otimes N}_\eps))$ if and only if there exists $\jupbold$ such that $(\bs\rhoup^\eps,\jupbold^\eps)\in\NCE_T^\eps$ and such that $\bs\A(\mu,\eta^\eps;\bs\rhoup^\eps,\jupbold^\eps)<\infty$ (cf. \cite[Proposition~2.31]{HPS21}). To each such $\jupbold^\eps$ we can associate a $\jhupbold^\eps$ such that $(\bs\rhoup^\eps,\jhupbold^\eps)\in\CE_T$ (proof for one species cf. \cite[Proposition~3.1]{EHS23}). By \cite[Proposition~3.3]{EHS23} we also have compactness, in the sense that there exists a sequence of pairs $(\bs\rhoup^\eps,\jhupbold^\eps)\in\CE_T$ and a pair $(\bs\rhoup,\jhupbold)\in\CE_T$ such that $\rho^{(i),\eps}_t\rightharpoonup\rho^{(i)}_t$ narrowly in $\calP(\Rd)$ for a.e. $t\in[0,T]$, $i=1,\ldots,N$ and such that $\int_\cdot\jh^{(i),\eps}_t\dx{t} \oset{\ast}{\rightharpoonup} \int_\cdot\jh^{(i)}\dx{t}$ weakly-$^\ast$ in $\calM((0,T)\times\Rd;\Rd)$ for $i=1,\ldots,N$.

    Next, we identify the limit the local limit of the $\eps$-dependent graph gradient structures as $\eps\to0$. To this end, we consider the first variation of the map $\vup\mapsto\tA(\mu,\eta^\eps;\rhoup,\vup)$, which is given by
    \begin{align*}
        \widetilde l^\eps_{\rhoup}(\vup)[\wup]\coloneqq\frac12\sum_{i=1}^N\iint_{G^{\!\!\:\eps}} w^{(i)}\eta^\eps\pra*{v^{(i)}_+\dx(\rho^{(i)}\otimes\mu)-v^{(i)}_-\dx(\mu\otimes\rho^{(i)})}.
    \end{align*}
    It turns out (proof for one species cf. \cite[Proposition~3.6]{EHS23}) that this nonsymmetric first variation is symmetric up to small errors, which vanish as $\eps\to 0$, and has a tensor structure in the sense that
    \begin{align*}
        \widetilde l^\eps_{\rhoup}(\babla\varphiup)[\babla\psiup]
        &=\sum_{i=1}^N\int_\Rd\nabla \varphi^{(i)}(x)\cdot\bbT^\eps(x)\nabla\psi^{(i)}(x)\dx\rho^{(i)}(x)+o(1), 
        \qquad\forall\varphiup,\psiup\in (C^2_c(\Rd))^N,
    \end{align*}
    with $\bbT^\eps \coloneqq \frac12\int_{\Rdx} (\cdot-y)\!\otimes\!(\cdot-y)\,\eta^\eps(\cdot,y)\dx\mu(y)\in C(\Rd;\R^{d\times d})$. Since for any compact $K\subset \Rd$ these $\eps$-tensors $\bbT^\eps$ converge strongly in $C(K;\R^{d\times d})$ as $\eps\to 0$ to $\bbT$ as in Theorem~\ref{thm:main_result} (cf. \cite[Proposition~3.8]{EHS23}), weak-strong convergence yields for any sequence $\rhoup^\eps \rightharpoonup\rhoup$ narrowly in $(\calP_2(\Rd))^N$ the limit (cf. \cite[Proposition~3.9]{EHS23})
    \begin{align*}
        \lim_{\eps\to 0} \widetilde l^{\eps}_{\rhoup^{\eps}}(\babla\varphiup)[\babla\psiup] &= \sum_{i=1}^N\int_\Rd \nabla \varphi^{(i)}\cdot\bbT\nabla\psi^{(i)}\dx\rho^{(i)}, 
        \qquad\forall\varphiup,\psiup\in (C^2_c(\Rd))^N.
    \end{align*}
    Employing in the graph framework a Cauchy-Schwarz type inequality and in the local framework Fenchel-Moreau duality, we obtain (proof for one species cf. \cite[Proposition~4.1]{EHS23}) for $\bs\rhoup^\eps \rightharpoonup\bs\rho$ narrowly in the sense from before, the lower semicontinuity
     $   \liminf_{\eps\to 0}\frac{1}{2}\int_0^T  \abs[\big]{\bra{\rhoup^\eps_t}'}_\eps^2\dx t\ge \int_0^T|\rhoup'_t|_\bbT^2\dx t$.
    Employing multiple truncation- and mollification arguments, the complexity of which is not increased by the presence of a second species in the variational derivative of the energy, we also obtain (proof for one species, cf. \cite[Proposition~4.2]{EHS23}) for $\bs\rhoup^\eps \rightharpoonup\bs\rho$ narrowly in the sense from before, the lower semicontinuity
        $\liminf_{\eps\to0}\calD_{\!\eps}(\rhoup^{\eps}) \ge \calD_{\;\!\bbT}(\rhoup)$.
    As the energy is narrowly continuous, this yields
        $0 = \liminf_{\eps\to0}\bs\calG_{\!\!\eps}(\bs\rhoup^\eps) \ge \bs\calG_\bbT(\bs\rhoup)$.
    For the local continuity equation, one has (proof for one species cf. \cite[Proposition~A.1]{EHS23})
    \begin{equation}\label{eq:chain-rule}
        \calE(\rhoup_t)-\calE(\rhoup_s)=\int_s^t\sum_{i,k=1}^N\int_\Rd\nabla K^{(ik)}*\rho^{(k)}_\tau(x)\dx j^{(i)}_\tau(x),
    \end{equation}
    being $(\jup_t)_{t\in[0,T]}\subset (\calM(\Rd;\Rd))^N$ such that $(\bs\rhoup,\jupbold)\in\CE_T$.
    
    This can be extended to $\bs\calG_\bbT(\bs\rhoup)\ge 0$ for all $\bs\rhoup \in \AC^2([0,T];((\calP_2(\Rd_{\bbT}))^N,W_\bbT^{\otimes N}))$ (proof for one species cf. \cite[Proposition~4.3]{EHS23}), thereby concluding the proof.
\end{proof}
Having established the link between the zero-level sets of the graph and local De Giorgi functionals, we are now in a position to show the existence of solutions to \eqref{eq:intro:NLIE:one}. This is done in the following theorem.
\begin{theorem}\label{thm:weak_sol_nlie_tensor}
Let $\mu,\vartheta$ satsify \eqref{mu1}, \eqref{mu2}, and \eqref{theta1}~--~\eqref{theta4}, respectively. Consider the tensor $\bbT$ defined in Theorem~\eqref{thm:main_result}. Assume $K^{(ik)}$, $i,k=1,\ldots,N$ satisfy \eqref{K1}~--~\eqref{K4}. Let $\varrhoup_0\in(\calP_2(\Rd))^N$ such that $\varrhoup_0\ll\mu$. There exists a weakly continuous curve $\bs\rhoup:[0,T]\to(\calP_2(\Rd))^N$ such that $\rhoup_t\ll\mu$ for all $t\in[0,T]$ which is a weak solution to \eqref{eq:nlnl-interaction-eq} with initial datum $\rhoup_0 = \varrhoup_0$, that is for the flux $\jhupbold\colon [0,T]\to(\calM(\Rd;\Rd))^N$ defined by
\begin{equation*}
  \dx \jh^{(i)}_t(x)=-\bbT\nabla\frac{\delta \calE}{\delta\rho^{(i)}}(x)\dx\rho^{(i)}_t(x),
\end{equation*}
we have $(\bs \rhoup,\jhupbold)\in\CE_T$.
\end{theorem}
\begin{proof}
    It is known (cf. \cite[Theorem~3.30]{HPS21} in conjunction with \cite[Appendix B]{EHS23}) that under the above assumptions for any $\eps>0$ there exists $(\bs\rhoup^\eps)_{\eps>0}\subset\AC^2([0,T];((\calP_2(\Rd))^N,\calT^{\otimes N}_{\eps}))$ with $(\rhoup_0^\eps)_\eps \subset (\calP_2(\Rd))^N$ such that $\sup_{\eps>0}\sum_{i=1}^N M_2(\rho_0^{(i),\eps}) < \infty$ with $\bs\calG_{\!\!\eps}(\bs\rhoup^\eps) = 0$. By Theorem \ref{thm:main_result}, this implies the existence of $\bs\rhoup \in \AC^2([0,T];((\calP_2(\Rd))^N,W_\bbT^{\otimes N}))$ with $\bs\calG_\bbT(\bs\rhoup) = 0$. Employing the chain-rule inequality \eqref{eq:chain-rule}, this is equivalent to (cf. \cite[Theorem~4.5]{EHS23}) $\bs\rhoup$ being a weak solution to \eqref{eq:intro:NLIE:one} in the above sense, thereby concluding the proof.
\end{proof}

\begin{acknowledgement}
  AE was supported by the Advanced Grant Nonlocal-CPD (Nonlocal PDEs for Complex Particle Dynamics: Phase Transitions, Patterns and Synchronization) of the European Research Council Executive Agency (ERC) under the European Union’s Horizon 2020 research and innovation programme (grant agreement No. 883363). GH acknowledges support of the German National Academic Foundation (Studienstiftung des deutschen Volkes) and the Free State of Saxony in the form of PhD scholarships. JFP thanks the Deutsche Forschungsgemeinschaft (DFG) for support via the Research Unit FOR 5387 POPULAR, Project No. 461909888. AS is supported by the Deutsche Forschungsgemeinschaft (DFG, German Research Foundation) under Germany's Excellence Strategy EXC 2044 -- 390685587, \emph{Mathematics M\"unster: Dynamics--Geometry--Structure}. 
\end{acknowledgement}

\vspace{\baselineskip}

\end{document}

%% file: header.tex
\usepackage[dvipsnames]{xcolor}
\usepackage{latexsym,amssymb}
\usepackage{amsmath,amsthm,amsxtra,amsbsy,accents,mathrsfs}
\usepackage{mathtools}
\usepackage{lmodern,microtype}
\usepackage{dsfont}
\usepackage[normalem]{ulem}

\usepackage{newtxtext} 
\usepackage{newtxmath} 

\usepackage{enumerate}
\usepackage{enumitem}

\usepackage{comment}

\usepackage[bookmarks=true, colorlinks=true]{hyperref}
\hypersetup{urlcolor=blue,citecolor=black,linkcolor=black}

\usepackage{aligned-overset}

\DeclareMathAlphabet{\mathbbmsl}{U}{bbm}{m}{sl}

\DeclareMathAlphabet{\mathup}{OT1}{\familydefault}{m}{n}
\newcommand{\dx}[1]{\mathop{}\!\mathup{d} #1}

\DeclarePairedDelimiter{\abs}{\lvert}{\rvert}
\DeclarePairedDelimiter{\norm}{\lVert}{\rVert}
\DeclarePairedDelimiter{\bra}{(}{)}
\DeclarePairedDelimiter{\pra}{[}{]}
\DeclarePairedDelimiter{\set}{\{}{\}}
\DeclarePairedDelimiter{\skp}{\langle}{\rangle}

\makeatletter
\newcommand{\customlabel}[2]{%
   \protected@write \@auxout {}{\string \newlabel {#1}{{#2}{\thepage}{#2}{#1}{}} }%
   \hypertarget{#1}{}
}
\makeatother

\makeatletter
\newcommand{\oset}[3][0ex]{%
  \mathrel{\mathop{#3}\limits^{
    \vbox to#1{\kern-2\ex@
    \hbox{$\scriptstyle#2$}\vss}}}}
\makeatother
\numberwithin{figure}{section}

\numberwithin{equation}{section}

\def\bbone{{\mathds{1}}}

\newcommand{\bs}[1]{{\boldsymbol{#1}}}

\newcommand{\R}{{\mathds{R}}}
\newcommand{\Rd}{{\R^d}}
\newcommand{\Rdzero}{{\Rd\setminus\set*{0}}}
\newcommand{\Rdx}{{\Rd\setminus\set*{x}}}
\newcommand{\Rddiag}{{\R^{2d}_{\!\scriptscriptstyle\diagup}}}

\newcommand{\jh}{{\hat{\jmath}}}

\newcommand{\jhup}{{\hat{\textnormal{\j}}}}
\newcommand{\jhupbold}{{\hat{\textnormal{\textbf{\j}}}}}

\newcommand{\babla}{\overline{\nabla}}

\newcommand{\dgrad}{\overline{\nabla}}

\newcommand{\eps}{\varepsilon}

\DeclareMathOperator{\supp}{supp}

\DeclareMathOperator{\CE}{CE}
\DeclareMathOperator{\NCE}{NCE}
\DeclareMathOperator{\AC}{AC}

\DeclareMathOperator{\Id}{Id}

\newcommand{\A}{{\mathcal{A}}}

\newcommand{\tA}{{\widetilde{\mathcal{A}}}}

\newcommand{\Ctheta}{{C_{\mathsf{mom}}}}
\newcommand{\ctheta}{{c_{\mathsf{nd}}}}

\newcommand{\Csupp}{{C_{\mathsf{supp}}}}

\newcommand{\Mplus}{\calM^+}

\makeatletter
\newcommand{\anabla}{{\mathpalette\a@nabla\relax}}
\newcommand\a@nabla[2]{%
	\setbox\z@=\hbox{$\m@th#1\bigtriangledown$}%
	\ht\z@.7\ht\z@
	\raise\dp\z@\box\z@
}
\makeatother

\makeatletter
\newcommand{\subalign}[1]{%
  \vcenter{%
    \Let@ \restore@math@cr \default@tag
    \baselineskip\fontdimen10 \scriptfont\tw@
    \advance\baselineskip\fontdimen12 \scriptfont\tw@
    \lineskip\thr@@\fontdimen8 \scriptfont\thr@@
    \lineskiplimit\lineskip
    \ialign{\hfil$\m@th\scriptstyle##$&$\m@th\scriptstyle{}##$\hfil\crcr
      #1\crcr
    }%
  }%
}
\makeatother

\def\calA{{\mathcal A}}  
\def\calD{{\mathcal D}} \def\calE{{\mathcal E}} 
\def\calG{{\mathcal G}}  
  
\def\calM{{\mathcal M}}  
\def\calP{{\mathcal P}}  
 \def\calT{{\mathcal T}}

  \def\scrL{{\mathscr  L}}

\def\bbD{{\mathbb D}}

 \def\bbN{{\mathbb N}} 
  
 \def\bbT{{\mathbb T}}

\newcommand{\jup}{\textnormal{j}}
\newcommand{\jupbold}{\textnormal{\textbf{j}}}

\newcommand{\vup}{\textnormal{v}}
\newcommand{\vupbold}{\textnormal{\textbf{v}}}
\newcommand{\wup}{\textnormal{w}}